\documentclass[10pt]{article}
\usepackage{extarrows}
\usepackage{latexsym,amsmath,amsthm,bbm}  
\usepackage{pgfplots}
\usepackage{pgfplotstable}
\pgfplotsset{compat=1.8}
\usepackage{mleftright}
\usepackage{booktabs}
\usepackage{mathtools}
\usepackage{graphicx}

\usepackage{amssymb}
\usepackage{caption}
\usepackage{subcaption}
\usepackage{tikz}
\makeatletter
\renewcommand{\vec}[1]{\boldsymbol{#1}}
\providecommand{\boldsymbol}[1]{\mbox{\boldmath $#1$}}
\makeatother

\newtheorem{theorem}{Theorem}[section]  
\newtheorem{lemma}[theorem]{Lemma}

\newtheorem{corollary}[theorem]{Corollary}  
\newtheorem{prop}[theorem]{Proposition}
\newtheorem{proposition}[theorem]{Proposition}
  
\newtheorem{remark}[theorem]{Remark}

\newcommand{\opL}{\mathcal{L}}

\newcommand{\R}{\mathbb{R}}

\newcommand{\N}{\mathbb{N}}

\newcommand{\Man}{\mathbb{M}}
\newcommand{\Sph}{\mathbb{S}}
\newcommand{\sph}{\mathbb{S}}
\newcommand{\NN}{\mathbb{N}}

\newcommand{\M}{\mathbb{M}}
\newcommand{\K}{\mathbf{K}}

\newcommand{\rank}{\operatorname{rank}}
\newcommand{\dist}{\operatorname{dist}}

 \newcommand{\dif}{\mathrm{d}}
\newcommand{\MM}{\mathbf{M}} 
\newcommand{\RR}{\mathbf{R}} 
\newcommand{\LL}{\mathsf{L}} \newcommand{\PP}{\mathsf{P}}
 \newcommand{\scal}{\mathcal H}
\newcommand{\bb}{\mathbf{b}} \newcommand{\Exp}{\mathrm{Exp}}
\def\ultra#1#2{P^{(#1)}_{#2}} 
 \renewcommand{\b}{\mathbf{b}}
\newcommand{\bbar}{\overline{\mathbf{b}}}
\newcommand{\Manoa}{M\=anoa}
\newcommand{\Hawaii}{Hawai\kern.05em`\kern.05em\relax i}
\numberwithin{equation}{section}

\title{Extending Data to
  Improve Stability and Error Estimates Using Asymmetric Kansa-like
  Methods to Solve PDEs}

\date{July 15, 2025}
\author{ T.~Hangelbroek \thanks{Department of Mathematics, University
    of \Hawaii, \Manoa, Honolulu, HI 96822, USA.  Research supported
    by by grants DMS-1716927 and DMS-2010051 from the National Science
    Foundation.}, F. J.~Narcowich\thanks{ Department of Mathematics,
    Texas A\&M University, College Station, TX 77843, USA. Research
    supported by grant DMS-1514789 from the National Science
    Foundation.}, J. D.~Ward\thanks{ Department of Mathematics, Texas
    A\&M University, College Station, TX 77843, USA.
  }}
\begin{document}

\maketitle

\begin{abstract} In this paper, a theoretical framework is presented
  for the use of a Kansa-like method to numerically solve elliptic
  partial differential equations on spheres and other manifolds. The
  theory addresses both the stability of the method and provides error
  estimates for two different approximation methods. A Kansa-like
  matrix is obtained by replacing the test point set $X$, used in the
  traditional Kansa method, by a larger set $Y$, which is a norming
  set for the underlying trial space. This gives rise to a rectangular
  matrix. In addition, if a basis of Lagrange (or local Lagrange)
  functions is used for the trial space, then it is shown
  that the stability of the matrix is comparable to the stability of
  the elliptic operator acting on the trial space.  Finally, two
  different types of error estimates are given. Discrete least squares
  estimates of very high accuracy are obtained for solutions that are
  sufficiently smooth. The second method, giving similar error
  estimates, uses a rank revealing factorization to create a
  ``thinning algorithm'' that reduces $\#Y$ to $\#X$. In practice,
  this algorithm doesn't need $Y$ to be a norming set.
\end{abstract}

\section{Introduction}
Asymmetric collocation, known as Kansa's method, is an often used
kernel-based mesh-free method for solving PDEs, even one subject to
boundary conditions. A review and discussion of the method is given in
\cite{Fasshauer_book}. 

The version of the problem considered here is for an elliptic
differential equation
\begin{equation}\label{elliptic}
\opL u= f
\end{equation}
on the sphere $\Sph^d$, where $f$ is smooth, and the operator $\opL$
is described in Section~\ref{opL}. (More generally, one could deal
with a similar problem on a smooth, compact Riemannian manifold $\M$.)
For a positive definite or a strictly conditionally positive definite
kernel $\Phi: \Sph^d\times \Sph^d \to \R$ and point sets
$X=\{x_1,\ldots,x_N\},Y=\{y_1,\ldots,y_M\}\subset \Sph^d$, Kansa's
method finds a $v$ in a kernel network
\begin{equation}
\label{Kansa}
v\in S_X(\Phi) = 
\mathrm{span}\{\Phi(\cdot,x_j)\mid x_j\in X\}
\qquad
\text{satisfying }
\quad
\opL v|_Y = f|_Y. 
 \end{equation}
 (For the case of a strictly conditionally positive definite $\Phi$
 the network is somewhat different, and is defined in
 \eqref{SCPD_net}.)  In practice this requires finding a solution of a
 linear system like $ \K a = f|_Y, $ where
 $\K:= \bigl( \opL^{(1)} \Phi(y_j,x_k)\bigr)$ is known as the Kansa
 matrix (this can be adjusted, for instance, by choosing a different
 basis for $S_X(\Phi)$ -- a modification we will consider below). The
 vector of coefficients $a$ is then used to generate the function
 $v=\sum a_j \Phi(\cdot,x_k)\in S_X(\Phi)$.

 In general, even if $\K$ is square, for instance by choosing $X=Y$,
 it is not necessarily symmetric positive definite, in contrast to the
 standard collocation matrix obtained from $\Phi$.  This provides a
 numerical method that is fairly easy to implement, but suffers from
 being potentially highly unstable.  Using carefully manufactured
 point sets, the matrix $\K$ may even be singular as shown in
 \cite{schaback_hon_2001}.


Closely related, especially for elliptic problems, 
are \emph{kernel differentiation methods}, where one seeks to solve
 \begin{equation}
 \label{FD}
 \MM V = f|_Y.
 \end{equation} 
 $\MM$ is the associated \emph{kernel differentiation matrix} and it
 has the form $\MM =\bigl(\opL \chi_j(x_k)\bigr)$. The $\chi_j$'s form
 a Lagrange basis for $S_X(\Phi)$ -- i.e.,
 $\chi_j(x_k)=\delta_{j,k}$. The matrix itself is a kernel collocation
 matrix. The matrix appears in pseudo-spectral methods and was
 recently discussed in \cite{EHNRW}.  A local variant is used in
 kernel (FD) and (RBF-FD) methods \cite{Oleg,HRW_25,larsson_21}. For a
 full stencil, it is known that $\MM=\K \Phi^{-1}$, where $\Phi$ is
 the standard collocation matrix $\Phi = (\Phi(x_j,x_k))$ on $X$.  The
 relation between solutions $a\in \R^X$ of (\ref{Kansa}) and
 $V\in\R^X$ of (\ref{FD}), is that $V=v|_X$ is the restriction of the
 kernel network $v =\sum a_j \Phi(\cdot, x_j)$ to $X$.


In both (\ref{Kansa}) and (\ref{FD}), the convergence of the computed
solution to the true solution is a consequence of stability and
consistency. Because kernel interpolation enjoys robust Sobolev error
estimates, consistency of these methods, measured as
$\| \opL u - \opL \mathcal{I}u\|$, is quite favorable.  Stability of
the method is measured as $\|\K^{-1}\|$ in (\ref{Kansa}) and
$\|\MM^{-1}\|$ in (\ref{FD}), in some matrix norm.  As mentioned
above, this is potentially problematic.  In short, the challenge to
proving convergence lies in the instability of the respective method,
which can be further identified as the inherent instability of the
Kansa matrix.

Under certain circumstances, the Kansa method can be shown to be
stable with $Y=X$: in \cite{EHNRW} it is shown that for an SBF kernel
and a Helmholtz operator (i.e., operators of the form
$\opL = c-\Delta$), $\K$ is invertible, with control on $\|\K^{-1}\|$.
A more general, but similar condition is used throughout \cite{Oleg}
-- in both cases, the requirement amounts to the fact that
$\opL^{(1)}K(\cdot,\cdot)$ is a kernel matrix.  Thus there are
instances where the Kansa problem is stable, but these require
compatibility between the kernel and operator.

This is in sharp contrast to other kernel-based mesh-free methods.
For example Galerkin methods \cite{NRW,Pat_diss} require only
coercivity of the bilinear form, a condition which is independent of
the kernel. Another kernel method
\cite{fasshauer_1999,narcowich_herm_1995} interpolates data of the
form $\lambda_if=d_i$, where the $\lambda_i$'s are linear functionals,
which could involve differential operators or point evaluations. The
idea is to apply these to a kernel and invert a collocation matrix
that is positive definite. However, this solves a problem where both
the values of the operator $\opL f$ and $f$ are known on $X$. This
data is not available for the Kansa method.

An underlying goal for this article, is to provide a method to
stabilize the Kansa method, thereby allowing the treatment of problems
like (\ref{elliptic}) where the operator $\opL$ and the kernel $\Phi$
are independent of one another.

\subsection{Overview} 
We consider a modification to the asymmetric collocation problem which
ensures stably invertible Kansa matrices by making a careful selection
$Y\neq X$. Specifically, we will need $Y$ to be a norming set that
satisfies conditions for $N$-dimensional spaces $W_N$ that satisfy a
Bernstein inequality. (See Section~\ref{norming_set_Y}.)  This fact is
important for our error estimates in both Sections \ref{soln_via_ls}
and \ref{thinning algorithm}.

The contributions of this article can be summarized as follows:

\begin{itemize}
\item We develop methods to construct an over-sampled (i.e.,
  rectangular) kernel collocation matrix, which generalizes the
  standard Kansa matrix. This new matrix is bounded below, and is
  independent or nearly independent of the problem size.
\item We provide approximation schemes to effectively treat the
  resulting over-determined systems, along with very good
  approximation rates. (See Theorems \ref{L2_est_L_2_proj} and \ref
  {thinning}.)
\end{itemize}

\paragraph{Constructing the over-sampled Kansa matrix $\K$.} 
The construction of the over-sampled Kansa matrix is made in two
stages: first, an $L_2$ norming set
$Y=\{y_k \mid k\le M\}\subset \Sph^d$ is chosen for the asymmetric
kernel space
$ S_X(\opL^{(1)}\Phi) = \mathrm{span}\{\opL \Phi(\cdot,x_j)\mid x_j\in
X\}$; the oversampled Kansa matrix
\begin{equation}\label{kansa_matrix}
  \K_X = \Bigl( \opL B_j(y_k)\Bigr)_{y_k\in Y}
\end{equation}
is assembled by employing a Riesz basis (see \eqref{Riesz})
$\{B_j\mid j\le N\}$ for the original kernel space $ S_X(\Phi)$. A
definite advantage in using the the basis $(B_j)$ is that for
strictly conditionally positive definite SBFs (see Section~\ref{SBFs})
there is no need to track side conditions or what happens in an
auxiliary space.

Construction of norming sets for $L_p$ in a more general setting,
namely for function spaces possessing a Nikolskii inequality, has
received significant recent interest in
\cite{DPSTT,limonova2022sampling,kashin2022sampling}.  The main
challenge for such problems is to produce a norming set $Y$ with
cardinality not much larger than that of $\dim S_X(\opL^{(1)}\Phi)$.
We show that this approach holds (namely, that $S_X(\opL^{(1)}\Pi)$
enjoys a Nikolskii inequality), while also presenting a different,
direct construction that uses a Bernstein inequality.

\paragraph{Solving $\K \vec{a} = f|_{Y}$.}
To treat the system $\K \vec{a} = f|_{Y}$, 
we study two approaches: 
\begin{itemize}
\item we show that discrete least squares provides, under further
  assumptions on the kernel, is a reasonably stable approximate solution.
  Such problems for $\RR^d$ are discussed by Cheung
  et~al.\ in \cite{Schaback-et-al}, and in several follow-up papers
  for similar problems on manifolds or surfaces
  \cite{Cheung_Ling_2018, Chen_Ling_2020, Chen_Ling_2025, Oleg};
\item we consider {\em thinning} the norming set to produce a subset
  $\tilde{Y}\subset {Y}$ having cardinality $\#\tilde{Y}=\#X$, while
  allowing $\|\K^{-1}\|$ to be suitably bounded.  For this, we employ
  rank revealing QR factorization as considered in
  \cite{gu1996efficient} -- because this step is fairly independent of
  the construction of the norming set, it may be possible to improve
  its performance by considering a faster implementation (this is
  remains an active field in numerical linear algebra
  \cite{Demmel,civril2009selecting,DuerschGu}) or by considering
  alternative thinning algorithms.
\end{itemize}
\smallskip

%
%
%
\section{Background}

Most of what we do here will be for $\sph^d$, which is the unit sphere
in $\R^{d+1}$. Even so, many of the results we obtain here will apply
to the more general setting of a manifold.  

\subsection{Manifolds \& Sobolev spaces}\label{sobolev_centers_nets}
Let $\M$ denote a $C^\infty$ compact Riemannian manifold without
boundary (i.e., closed), and having bounded geometry; see
\cite[Sect.~7.2.1]{triebel1992} The Sobolev space $W_p^{k}(\M)$ is
defined via the covariant derivative $\nabla^k$, which takes functions
to tensor fields of covariant order $k$. The norm is defined as
$\|f\|_{W_p^{k}(\M)}^p = \sum_{j=0}^k \int \|\nabla^j f(x)\|_x^p\dif
x$.  Here the norm for a tensor of covariant order $j$, denoted
$\|T\|_x$, is the norm induced by the Riemannian metric. See
\cite{HNW} for a complete description.

In the case where we are dealing with $p=2$, we may use the norm
equivalence between $W_2^{k}(\M)$ and the potential space $H^k(\M)$
\cite{strichartz_83}, which has the norm
\begin{equation}\label{pot_space_LL}
  \|f\|_{H^k(\M)}:=\|\LL^k f\|_{L_2(\M)},\ \text{where } \LL=
  \sqrt{\lambda_d - \Delta} \ \text{and }
  \lambda_d=\frac{d-1}{2}.
\end{equation}
Besides being easier to work with, the potential spaces provide a
simple way to deal with fractional Sobelev spaces; namely,
$H^s=\|\LL^s f\|_{L_2(\M)}$, $s\in \R$.  Potential spaces can also be
defined for $1\le p \le \infty$. They are denoted by $H_p^s$. However,
they are equivalent to the $W_p$'s only for $1<p<\infty$.

\paragraph{Centers in $\M$}
Define $\b(x,r)$ be the open ball of radius $r$ centered at $x\in \M$
and $\bbar(x,r)$ to be its closure. Let $X$ be a finite set of
distinct points in $\M$; we will call these the \emph{centers}. For
$X$, we define these quantities: \emph{mesh norm}, or \emph{fill
  distance}, $h_X=\sup_{y\in \sph^n} \inf_{\xi\in X} d(\xi,y)$, where
$d(\cdot,\cdot)$ is the geodesic distance between points on the
sphere; the \emph{separation radius},
$q_X=\frac12 \min_{\xi\neq\xi'}\, d(\xi,\xi')\,$; and the \emph{mesh
  ratio}, $\rho_X:=h_X/q_X\ge 1$. Of course, we may use other sets of
centers, $Y$, $Z$ and so on. If $\rho$ is bounded, and not large, then
we say that the point set $X$ is quasi-uniformly distributed, or
simply that $X$ is quasi uniform.

Geometrically, for every $x\in \M$ there will be some $\xi\in X$ such
that $x\in \bbar(\xi,h_X)$. Consequently,
$\M =\cup_{\xi\in X}\bbar(\xi,h_X)$; \emph{i.e.}, the union is a
covering for $\M$. However, for $r< h_Y$,
$\cup_{\xi\in X}\bbar(\xi,r)$ doesn't cover $\M$. The interpretation
of separation radius $q_X$ is that there is at least one pair of
closed balls $\bbar(\xi,q_X)$ and $\bbar(\eta,q_X)$ which intersect in
a single point. This fails for any pair with
$\frac12 \dist(\xi',\eta')<q_X$ .

\paragraph{Minimal $\epsilon$ nets in $\M$}
We will need another tool, \emph{minimal $\epsilon$
  nets}\footnote{These go by other names; $\epsilon$ nets, for
  example.} The description for them is given in
\cite[Sect.~3]{grove-petersen-1988-1}. Let $\epsilon>0$. There exists
an ordered set of points $\{p_1,\ldots,p_N\}\subset \M$ such that the
$\cup_{j=1}^N \b(p_j,\epsilon) = \M$ and such that the balls
$\b(p_j,\epsilon/2)$ are disjoint. Such a set is called a
\emph{minimal $\epsilon$-net} in $\M$\footnote{An $\epsilon$-net is a
  set of points $X =\{p_1, \dots, p_N\}$ for which $\bigcup \b(p_j, \epsilon)$
  covers $\M$ -- in other words, for which $h(X,\M) < \epsilon$. Also,
  these nets are quasi uniform, with separation distance
  $q\ge \epsilon/2$ and mesh ratio $h/q\le2$.}.
It has the following two important properties: First, there is a
number $N_1=N_1(\varepsilon, \M)$ for which $N\le N_1$. Second, there
exists an integer $N_2=N_2(\M)\ge 1$ such that for any $p\in\M$ the
open ball $\b(p,\epsilon)$ \emph{intersects at most} $N_2$ of the
balls $\b(p_j,\epsilon)$. It is remarkable that $N_2$ is
\emph{independent} of $\epsilon$ and, in fact, depends only on general
properties of $\M$ itself. It is important to note that such nets can
be constructed numerically \cite{Gonz_1985,Gonz_2010}.

\subsection{The operator $\opL$}\label{opL}
The operator $\opL$ in equation \eqref{elliptic} is assumed to have
$C^\infty$ coefficients in any local chart, and that in such a chart
$\opL$ is a uniformly strongly elliptic second order differential
operator. Also, $\opL$ satisfies the additional assumption
\begin{equation}\label{op_lower_bound_1}
\|\opL f\|_{L_2(\M)} \ge c_{\opL} \|f\|_{L_2(\M)}
\end{equation}
We will need the following result, which was proved in
\cite[Proposition 5.2 \& Remark 5.3]{NRW}

\begin{prop}
\label{regularity}
Let $\opL$ be as described above. If $f$ is a distributional solution
to $\opL f=g$, where $g\in H^s(\Man)$, $0\le s$, $s\in \R$, then
$f\in H^{s+2}(\Man)$. In addition, for any $t<s+1$ there is a constant
$C_t>0$ such that ,
$\|f\|_{H^{s+2}(\Man)} \le C_t (\| \opL f\|_{H^{s}(\Man)}+
\|f\|_{H^t(\Man)})$ and
$\|f\|_{H^{s+2}(\Man)} \le C\|\opL f\|_{H^s(\Man)}$ all hold.
\end{prop}

For our purposes, we will take $s=0$ and use the fact that
$H^k(\Man)=W_2^k(\Man)$. Since $\opL$ is second order differential
operator we have
$\|\opL f \|_{W_2^k(\Man)} \le C \|f\|_{W^{k+2}_2(\Man)}$. The
equivalence of $H^k$ and $W_2^k$ imply that
$\|\opL f \|_{H^k} \le C \|f\|_{H^{k+2}}$. Putting this together with
the inequality for $s=0$ in the proposition above, we have
\begin{equation}
\label{equiv}
\begin{cases}
  \|\opL f\|_{H^k(\Man) }\le \Gamma_1 \|f\|_{{H^{k+2}}(\Man)}\\
  \|f\|_{H^{k+2}(\Man)} \le \Gamma_2 \|\opL f\|_{H^k(\Man)}.
 \end{cases}
\end{equation}

The set of equations imply that $\opL: H^k \to H^{k+2}$ and
$\opL^{-1}:H^{k+2}\to H^k$ are both bounded.

\section{Spherical Basis Functions} \label{SBFs}

Let $\{Y_{\ell,m}:\ell=0,\ldots,\infty; m=0\ldots N_{\ell,d} \}$ be the
set of (real) spherical harmonics on $\sph^d$
\cite{Mueller-66-1,Stein-book}, where $N_{\ell,d}$ is the dimension of
the space of order $\ell$ spherical harmonics, which we denote by
$\scal_\ell$. Together, these form an orthonormal basis for
$L_2(\sph^d)$. Spherical harmonics are eigenfunctions of the
Laplace-Beltrami operator $\Delta$ on $\sph^d$. The eigenvalues of
$-\Delta$ are $\lambda_\ell = \ell(\ell + d-1)$. The eigenspace
corresponding to $\lambda_\ell$ is degenerate, and has dimension
\begin{equation}
\label{dim_ell}
N_{d,\ell} = 
\left\{
\begin{array}{cc}
1,&\ell=0,\\[6pt]
\displaystyle{\frac{(2\ell+d-1) 
\Gamma(\ell+d-1)}{\Gamma(\ell+1)\Gamma(d)}} \sim \ell^{d-1}\,,
&\ell\ge 1\,.
\end{array}\right. .
\end{equation}
A zonal function is a rotationally invariant kernel of the form
\begin{equation}\label{ultra_sph}
  Z(x \mathbf \cdot y):= \sum_{\ell=0}^\infty\widehat Z_\ell
  \frac{\ell+\lambda_d}{\lambda_d\omega_d}\ultra {\lambda_d} \ell
  (x{\mathbf\cdot} y), \ \text{where } \ultra {\lambda_d} \ell
  (x{\mathbf\cdot} y)= \frac{\lambda_d\,\omega_d}{\ell + \lambda_d}
  \sum_{m=0}^{N_{d,\ell}} Y_{\ell,m}(x) Y_{\ell,m}(y),\ \lambda_d= 
  \frac{d-1}{2}.
\end{equation}
Here, $\ultra {\lambda_d} \ell (\cdot)$ the ultraspherical polynomial
of order $\lambda_d$ and degree $\ell$.

We can now define a positive definite \emph{Spherical Basis Function}
(SBF). It is a zonal function in which all of the $\widehat Z_\ell$'s
are positive. A \emph{Strictly, Conditionally, Positive Definite
Function} (SCPD) of order $L$ is a zonal function for which
$\widehat Z_\ell>0$ for $\ell\ge L$ and is either $0$ or negative when
$\ell=L-1$

Take $\PP_\ell$ to be the orthogonal projection of $L_2(\sph^d)$ onto
$\scal_\ell$ and consider the operator
$\LL=\sqrt{\lambda_d-\Delta}=
\sum_{\ell=0}^\infty(\ell+\lambda_d)\PP_\ell$ defined in
\eqref{pot_space_LL}. It is easy to show that the kernel of $\PP_\ell$
is given by
$P_\ell(x\mathbf \cdot y) = \frac{\ell +
  \lambda_d}{\lambda_d\,\omega_d}\ultra {\lambda_d} \ell
(x{\mathbf\cdot} y).$ We may use this to define a particularly
important class of kernels to be used here.

Let $\beta>0$ and let $G_{\beta}$ be the fundamental solution to
$\LL^{\beta}G_\beta=\delta$; $G_{\beta}$ is a zonal kernel with an
expansion in ultraspherical polynomials having coefficients
$\widehat G_\beta(\ell)=(\ell+\lambda_d)^{-\beta}$:
\begin{equation}\label{ultraspherical_SBF}
  G_\beta(x\mathbf \cdot y)= \sum_{\ell=0}^\infty \widehat G_\beta(\ell)
  \frac{\ell+\lambda_d}{\lambda_d\omega_d}\ultra {\lambda_d} \ell
  (x{\mathbf\cdot} y), \ x,y \in S^d,
\end{equation}
This kernel is a positive definite SBF. Another related kernel, also a
positive definite SBF, is
\begin{equation}\label{G_psi}
  \Psi_\beta:=G_\beta+ G_\beta\ast \psi, \ \widehat \Psi_\beta =
  (\ell+\lambda_d)^{-\beta}(1+\widehat \psi(\ell))
\end{equation}
where $\psi\in L_1$ satisfies $\widehat \psi(\ell)+1>0$. These SBFs
were discussed in detail in \cite[Section 2.3]{MNPW}.

We will need to strengthen the Bernstein inequality in \cite[Theorem
6.1]{MNPW}, which states that for $g$ in the SBF network
$S_X(\Psi_\beta):= \{\sum_{\xi\in X}c_\xi\Psi_\beta((\ )\mathbf \cdot \xi)\}$,
\begin{equation}\label{thm6.1_MNPW}
\|g\|_{H^\gamma_p(\sph^d)}\le Cq^{-\gamma}\|g\|_{L_p(\sph^d)},
\end{equation}
provided $0<\gamma <\beta-d/p'$ and $1\le p \le \infty$.

\begin{proposition}\label{soblev_bernstein}
  Let $g_\beta:=\sum_{\xi\in X}c_\xi\Psi_\beta((\ )\mathbf \cdot \xi)$
  and suppose that $\gamma>0, \varepsilon>0$ satisfy
  $\gamma+\varepsilon <\beta-d/p'$. Then
\begin{equation}\label{L_p_sob_bern}
  \|g_\beta\|_{H^{\varepsilon+\gamma}_p}\le
  Cq^{-\gamma}_X\|g_\beta\|_{H^{\varepsilon}_p}
\end{equation}
\end{proposition}

\begin{proof}
From \eqref{G_psi}, we see that $\LL^{\gamma+\varepsilon}\Phi_\beta
= \LL^{\gamma}\Phi_{\beta+\varepsilon}$. Consequently,
$\LL^{\gamma+\varepsilon}g_\beta=\LL^\gamma g_{\beta+\varepsilon}$, and so
\[
  \|g_\beta\|_{H_p^{\gamma+\varepsilon}}=
  \|g_{\beta+\varepsilon}\|_{H_p^{\gamma}}\le Cq_X^{-\gamma}
  \|g_{\beta+\varepsilon}\|_{L_p}=Cq^{-\gamma}_X\|\LL^\varepsilon g_{\beta}\|_{L_p}=
  Cq^{-\gamma}_X\|g_{\beta}\|_{H^\varepsilon_p},
\]
which completes the proof.
\end{proof}

\begin{remark} \rm Later, we will need the special case in which
  $\gamma=1$, $\varepsilon=2$, $p=2$ and $3+d/2<\beta$:\rm
\begin{equation}\label{special_bern}
  \|g_\beta\|_{H^3}\le Cq^{-1}_X\|g_\beta\|_{H^2}
\end{equation}
\end{remark}

For $p=2$, we can extend the results in
Proposition~\ref{soblev_bernstein} to certain strictly conditionally
positive definite SBFs (SCPDs) of order $L$. These include the
thin-plate splines restricted to $\sph^d$, which are defined in
\eqref{TPS}.

Let $\Pi_{L-1}$ be the set of all spherical harmonics of degree $L-1$
or less. A kernel $\phi(x\mathbf \cdot y)$ is said to be SCPD if the
collocation matrix
$A=(\phi(\xi_i\mathbf \cdot \xi_j)_{\xi_i,\xi_j\in X}$ is positive
definite when restricted to the span of all $c\in \R^{|X|\times |X|}$
satisfying $\sum_{\xi\in X} c_\xi p(\xi)= 0 \ \forall p\in
\Pi_{L-1}$. If $L\ge 1$ is the smallest integer for which this
condition holds, $\phi$ is said to have \emph{order} $L$. (If $L=0$,
$\phi$ is a positive definite SBF.) The network\footnote{Since the
  order $L$ is unique, $L$ is implicit in $S_X(\phi)$. No extra
  notation is needed.}for an SCPD of order $L$ is defined by
\begin{equation}\label{SCPD_net}
    S_X(\phi)=\{\sum_{\xi\in X} c_\xi \phi((\ ) \mathbf \cdot\xi):
  \sum_{\xi\in X} c_\xi p(\xi)= 0 \ \forall p\in \Pi_{L-1} \text{ and
  } \xi\in X\} \cup \Pi_{L-1}.
\end{equation}
Another way to define the order is to look at the expansion of $\phi$
in a basis of ultraspherical polynomials. If $\widehat \phi(\ell)>0$
for all $\ell \ge L$, but is $0$ or negative for $\ell = L-1$, then
the order is $L$. 

Our aim is to obtain Bernstein inequalities for a special class of
SCPD kernels. Suppose $\Psi_\beta$ is given by \eqref{G_psi}. If for
all $\ell\ge L$ the SCPD kernel $\phi$ satisfies
$\widehat \phi(\ell)= \widehat \Psi_\beta(\ell)$, then we will say
$\phi$ is a $\beta$-class SCPD kernel of order $L$. Since $\phi$ and
$\Psi_\beta$ differ in their ultraspherical expansions only for
$\ell\le L-1$, we hve that $\phi-\Psi_\beta=p_{L-1}$, where
\begin{equation}\label{poly_term}
  p_{L-1}(x{\mathbf\cdot} y)= 
  \sum_{\ell=0}^{L-1} b_\ell \ultra {\lambda_d} \ell
  (x{\mathbf\cdot} y).
\end{equation}

There are several properties that will be useful. We collect these in
the Lemma below.
\begin{lemma}\label{SCPD-beta}
    Let the $c_\xi$'s satisfy the condition in \eqref{SCPD_net}. Then
  $\sum_{\xi\in X} c_\xi p_{L-1}(x{\mathbf\cdot} \xi) = 0$,
  $\forall x\in \sph^d$. In addition, we have that
  $\sum_{\xi\in X}c_\xi \phi(x{\mathbf\cdot} \xi) = \sum_{\xi\in X}
  c_\xi \Psi_\beta(x{\mathbf\cdot} \xi)$. Finally, these two sums
  are orthogonal to $\Pi_{L-1}$ in all of the Sobolev spaces $H^\mu$,
  with $\mu\ge 0$.
\end{lemma}

\begin{proof}
Using \eqref{ultra_sph}, we have for every $x\in \sph^d$
\[
  \sum_{\xi\in X} c_\xi p_{L-1}(x{\mathbf\cdot} \xi) =
  \sum_{\ell=0}^{L-1} b_\ell\frac{\lambda_d\,\omega_d}{\ell +
    \lambda_d} \sum_{m=0}^{N_{d,\ell}} Y_{\ell,m}(x)\big\{\sum_{\xi
    \in X}c_\xi Y_{\ell,m}(\xi)\big\}.
\]
Since $Y_{\ell,m}$ has $\ell\le L-1$, the function above is in
$\Pi_{L-1}$. The condition on the $c_\xi$'s implies that
$\sum_{\xi\in X}c_\xi Y_{\ell,m}(\xi)=0$, which establishes that the
first sum is $0$. The second follows from this and the fact that
$\phi-\Psi_\beta=p_{L-1}$. To obtain the orthogonality, we examine the
calculation above. It shows that in the expansion of
$\sum_{\xi\in X}c_\xi \phi((\ ){\mathbf\cdot} \xi)$ in the
$Y_{\ell,m}$'s has nonzero coefficients only for $\ell\ge L$. Thus it
is orthogonal to $\Pi_{L-1}$ in $L_2$.  Similar argument yields the
result for $H^\mu$.
\end{proof}

We now turn to establishing a Bernstein inequality for $\beta$-class
SCPD kernels, one that is similar to the one in
Proposition~\ref{soblev_bernstein}. 
\begin{theorem}\label{general_L2 Bernstein}
  Let $\phi$ be a $\beta$-class SCPD kernel and let $g$ be in the
  network $S_X(\phi)$ defined in \eqref{SCPD_net}. Then $g$ satisfies
  the Bernstein inequality
  $\|g\|_{H_2^{\gamma+\varepsilon}}\le
  Cq_X^{-\gamma}\|g\|_{H^\varepsilon}$, where $\gamma>0$ and
  $\varepsilon>0$ satisfy $\gamma+\varepsilon<\beta-d/2$.
\end{theorem}  
\begin{proof}
  Since $g\in S_X(\phi)$, it has the form
  $g=\sum_{\xi\in X}c_\xi \phi((\ ){\mathbf\cdot} \xi) + Q$, where
  $Q\in \Pi_{L-1}$. By Lemma~\ref{SCPD-beta},
  $\sum_{\xi\in X}c_\xi \phi((\ ){\mathbf\cdot} \xi)=\sum_{\xi\in
    X}c_\xi \Psi_\beta((\ ){\mathbf\cdot} \xi)=:g_\beta$ is orthogonal
  to $\Pi_{L-1}$ in $H^\mu$, for all $\mu\ge 0$, and hence to
  $Q$. Thus $g=g_\beta +Q$ satisfies
  $\|g\|_{H^\mu}^2=\|g_\beta\|_{H^\mu}^2+ \|Q\|_{H^\mu}^2$.
  Proposition~\ref{soblev_bernstein} applies to $g_\beta$ and, if we
  slightly modify \cite{MNPW}[Theorem 4.19], to $Q$ as well. Letting
  $p=2$ we have that
  $\|g\|_{H^{\gamma+\varepsilon}}^2=\|g_\beta\|^2_{H^{\gamma+\varepsilon}}
  + \|Q\|_{H^{\gamma+\varepsilon}}^2$. Applying Bernstein
  inequalities from Proposition~\ref{soblev_bernstein} we have that
  $\|g_\beta\|^2_{H^{\gamma+\varepsilon}}\le Cq^{-2\gamma}
  \|g_\beta\|^2_{H^{\varepsilon}}$ and, after possibly adjusting the
  constant $C$,
  $\|Q\|^2_{H^{\gamma+\varepsilon}}\le
  Cq_X^{-2\gamma}\|Q\|^2_{H^{\varepsilon}}$. Adding these up and using
  the orthogonality of $g_\beta$ and $Q$, we have that
  $\|g\|_{H^{\gamma+\varepsilon}}^2\le
  Cq_X^{-2\gamma}\|g\|^2_{H^{\varepsilon}}$. Taking square roots then
  yields the Bernstein inequality that we wanted.
\end{proof}

The thin-plate splines\footnote{We are including what are others call
  potential splines as thin-plate splines. These are frequently
  treated as a separate category.} form one of the most important of the
classes of SCPD kernels. These are defined in \cite[Section
8.3]{Wendland-05-01}; their Fourier-Legendre coefficients are computed
in \cite[Section 4.2]{NSW_2007}, with a slightly different
normalization than we use here. The thin-plate splines themselves are
given below\footnote{Coefficients for $\ell\le s$ may be found in
  \cite{Baxter}.}.
\begin{equation}\label{TPS}
  \left.
    \begin{array}{l}
\displaystyle{
\phi_s(t)=
\left\{
\begin{array}{cc}
  (-1)^{\lceil (s)_+\rceil}(1-t)^s, & s > - \frac{d}{2},  \ s \not\in \NN\\[5pt]
  (-1)^{s+1}(1-t)^s\log(1-t), &  s\in\NN.
\end{array}
\right.}\\[18pt]
  \hat\phi_s(\ell)=
      C_{s,d}\frac{\Gamma(\ell-s)}{\Gamma(\ell+s+d)} \sim \ell^{2s+d} \sim
      \lambda_\ell^{s+d/2}. \ \ell>s,
\end{array}
\right\}
\end{equation}
where the factor $C_{s,d}$ is given by
\[
C_{s,d}:=2^{s+n}\pi^{\frac{d}{2}}\Gamma(s+1)\Gamma(s+\frac{d}{2}) 
\left\{
\begin{array}{ll}
 \frac{\sin(\pi s)}{\pi} & s > - \frac{d}{2},  \ s \not\in \NN\\[5pt]
1, & s\in\NN.
\end{array}
\right.
\]

When $s$ is an integer or half integer, if $\ell>s$, then
$\hat\phi_s(\ell)$ is analytic in $\ell$. In \cite[Section 3]{MNPW},
it was shown that
\[
  \hat\phi_s(\ell)=C_{s,n}\big(\widehat{G}_{2s+d}(1+\hat\psi(\ell))\big),\
  \ell>s,
\]
where $\psi\in L_1(\Sph^d)$. For $\ell\le s$ the Fourier-Legendre
coefficients for the thin-plate splines may be found in \cite[Section
2.3]{Baxter}. When $\ell\le s$, the coefficients for the $L_1$
function can be freely chosen, as long as they are non negative. By
applying Theorem~\ref{general_L2 Bernstein}, we obtain his result:

\begin{corollary}\label{TPS Bernstein}
  Let $S_X(\phi_s)$ be the network \eqref{SCPD_net}, with $\phi$
  replaced by $\phi_s$. Here $(-1)^{\lceil (s)_+\rceil}(1-t)^s$ holds
  for $s=k+\frac12$, and $(-1)^{s+1}(1-t)^s\log(1-t)$ holds for
  $s\in \NN$. Then $g\in S_X$ satisfies the Bernstein inequality
  $\|g\|_{H^{\gamma+\varepsilon}}\le
  Cq_X^{-\gamma}\|g\|_{H^\varepsilon}$, where $\gamma>0$ and
  $\varepsilon>0$ satisfy $\gamma+\varepsilon<\beta-d/2$.
\end{corollary}

%
%
%

\section{Norming sets }
\label{S:Norming}

Assume that $W_N\subset C(\M)$ is an $N$-dimensional subspace of
$C(\M)$.  We wish to find a point set $Y \subset \M$ which serves as a
norming set for $W_N$ equipped with the $L_p(\M)$ norm.  Specifically,
we seek conditions on $W_N$ for which
\begin{equation}\label{norming_set}
  \bigl(\forall w\in W_N\bigr)
  \qquad 
  \|w\|_{L_p(\M)}\le C_{\mathsf{N}}\left( \frac{1}{M}
    \sum_{y\in {Y}} |w(y)|^p
  \right)^{1/p}
\end{equation}
holds with cardinality $M:=\#{Y}$ not much larger than $N:=\dim(W_N)$.
Ideally $M\le CN$ for a global constant, and, importantly, for a set
$Y$ which is distributed quasi-uniformly.

\paragraph{Marcinkiewicz-Zygmund inequalities via Nikolskii inequalities}
The existence and construction of such sets has recently been
investigated in \cite{DPSTT,limonova2022sampling,kashin2022sampling}).
If the space $W_N$ enjoys the Nikolskii inequalities below for all
$w\in W_N$,
\begin{equation*}
\|w\|_{L_\infty(\M)} \le C_1 \sqrt{N}\|w\|_{L_2(\M)} \qquad
\text{and}
\qquad
\|w\|_{L_\infty(\M)} \le C_2 \|w\|_{L_{\log N}(\M)} ,
\end{equation*}
then by \cite[Theorem 2.2]{DPSTT} the Marcinkiewicz-Zygmund inequality
holds
\begin{equation}
\label{MZ}
(1-\epsilon) \|w\|_{L_p(\M)}^p
\le
 \frac{1}{M}\sum_{y\in {Y}} |w(y)|^p \le (1+\epsilon)
 \|w\|_{L_p(\M)}^p
\end{equation}
holds for all $w\in W_N$, where ${Y}\subset \M$ with 
 $M =\#\tilde{Y} \le  C N( \log N)^3 .$ 
 The lower bound in (\ref{MZ}) guarantees that  (\ref{norming_set})
 holds, although with 
 a large value of $M$ relative to $N$, 
 and without a guarantee of quasi-uniformity.

 Although the upper bound in (\ref{MZ}) is not relevant for our
 present purposes, it is worth mentioning that such estimates also
 play a role in kernel approximation (see \cite{Wenzel}, especially
 Theorem 7).
 
 \subsection{Norming sets via Bernstein inequalities}
 \label{norming_set_Y}
 We will show that if $\M$ is a compact Riemannian manifold without
 boundary and $W_N\subset C(\M)$ satisfies a suitable Bernstein
 inequality, then the norming set condition above follows,
 specifically (\ref{norming_set}) holds for a quasi-uniform set $Y$
 which satisfies $M\sim N$.

\begin{theorem}\label{Theorem:Norming_Set}
If $W_N$ is a space which satisfies the 
Bernstein inequality 
\begin{equation}
\label{Bern}
\|w\|_{W_p^{k}(\M)}\le C_{\mathfrak{B}} N^{k/d}\|w\|_{L_p(\M)}, \
\text{where }1\le p\le \infty,
\end{equation}
for all $w\in W_N$
then  there is a constant $\gamma >0$
so that for any  $Y\subset \M$ 
with $h_Y\le  \gamma  N^{-1/d}$ we have
\begin{equation}
\label{Norming_est_lemma}
\|w\|_{L_p(\M)} \le C_k
h_Y^{d/p} \| w\left|_Y\right.\|_{\ell_p(Y)} .
\end{equation}
In particular, it is possible to select a suitable norming set $Y$
which is quasi-uniform and has cardinality
$\# Y \le C_{\M} \rho^d 2^d\gamma^{-d} N$, where $\rho$ is the mesh
ratio and the constant $C_\M$ depends only on $\M$.
\end{theorem}
We call the constant  $\kappa:=\#Y/N$ the 
{\em degree of
oversampling}.
By the above result, we have $\kappa\le C_{\M} \rho^d \gamma^{-d}$.
\begin{proof}
  To get a norming set, we combine (\ref{Bern}) and
  Lemma~\ref{Lemma:Sampling} (proved in the section below) to obtain,
  for any function $w\in W_N$ (hence in $W_p^k(\M)$), that
\begin{equation*}
  \|w\|_{L_p(\M)} 
  \le \frac12 C_{k} \bigl( h_Y^{k} \|w\|_{W_p^{k}(\M)} +
  h_Y^{d/p} \| w\left|_Y\right.\|_{\ell_p(Y)}\bigr).
\end{equation*}
Applying the Bernstein inequality (\ref{Bern}) gives
\begin{equation*}
  \|w\|_{L_p(\M)} 
  \le
  \frac12 C_{k}C_{\mathfrak{B}} h_Y^{k} N^{k/d} \|w\|_{L_p(\M)} +
  \frac12 C_{k} h_Y^{d/p} \| w\left|_Y\right.\|_{\ell_p(Y)}
\end{equation*}
So if $C_{k}C_{\mathfrak{B}} (h_Y N^{1/d})^k\le 1$, then, upon
subtracting and multiplying by 2, we have
\[
  \|w\|_{L_p(\M)} \le C_{k} h_Y^{d/p} \|
  w\left|_Y\right.\|_{\ell_p(Y)}.
\]
This holds for any subset $Y$ with
\[
  \frac12 (C_{k} C_{\mathfrak{B}})^{-1/k} \, N^{-1/d} \le h_Y\le
  (C_{k} C_{\mathfrak{B}})^{-1/k} \, N^{-1/d}.
\]
The constant $\gamma$ may be chosen to be
$(C_{k} C_{\mathfrak{B}})^{-1/k}$, so that $h_Y \le \gamma N^{-1/d}$.
If in addition we select $Y$ so that
$h_Y \ge \frac12 (C_{k} C_{\mathfrak{B}})^{-1/k} \, N^{-1/d}=\frac12
\gamma N^{-1/d}$, and that $Y$ is quasi-uniform with mesh ratio
$\rho = h_Y/q_Y$, then
\[
\# Y \le C_{\M} (q_Y)^{-d}\le C_{\M} \rho^d (h_Y)^{-d} \le C_{\M}
\rho^d 2^d ( C_{k} C_{\mathfrak{B}})^{d/k} N= C_{\M}
\rho^d 2^d \gamma^{-d} N.
\]
So in this case, (\ref{Norming_est_lemma}) holds for a quasi-uniform
set $Y$ having cardinality on par with $N$.
\end{proof}

A set $Y$ can be constructed using the minimal $\epsilon$ nets
discussed in Section~\ref{sobolev_centers_nets}. We may choose points
$\{y_1,y_2 \ldots,y_M\} \subset X$ so that $\epsilon$ satisfies
\[
  \frac12 \sqrt[k]{C_{k} C_{\mathfrak{B}}} \, N^{-1/d} \le \epsilon
    \le \sqrt[k]{C_{k} C_{\mathfrak{B}}} \, N^{-1/d}.
\]
Since the minimal $\epsilon$ net is quasi uniform, and $\epsilon$ may
be chosen so that the inequality above is satisfied, we may choose $Y$
to be this $\epsilon$ net. Of course, this isn't the only possible
choice for $Y$. Clearly there are many others.

\subsection{Sampling inequalities for manifolds} 
The key to the result is a ``sampling estimate'' for $\M$ which
extends the Euclidean estimate in \cite[Theorem 3.5]{Madych}.
There have been a number of versions of sampling inequalities more
general than those in \cite{Madych}. For instance a version dealing
with fractional orders that works on Euclidean domains satisfying a
cone condition is given in \cite{Torres}.
\begin{lemma}
\label{Lemma:Sampling}
For any $k>d/2$ there are positive constants $\frac12 C_{k}$  
and $h_{k}$ so that
for any $f\in W_p^{k}(\M)$ and $Y\subset \M$ 
with $h_Y=h(Y,\M)\le h_{k}$
we have
\begin{equation}
\label{samp}
\|f\|_{L_p} \le \frac12
C_{k}\bigl(  (h_Y)^{k} \|f\|_{W_p^{k}} +
(h_Y)^{d/p} \| f\left|_Y\right.\|_{\ell_p(Y)}\bigr).
\end{equation}
\end{lemma}
\begin{proof}
Cover $\M$ by sets $\M =\bigcup_{j=1}^K \bb(P_j,R/\sqrt{d})$,
where $R$ is less than the injectivity radius of $\M$. Equip each $\bb(P_j,R)$ with normal coordinates
about $P_j$
given by the chart
$$\psi_j = (\Exp_{P_j})^{-1}:\bb(P_j,R)\to B(0,R).$$
where $Q_j=\Exp_{P_j}([-r,r]^d) $ 
with
$r\sqrt{d} <R$. Note that $\bb(P_j,R/\sqrt{d}) \subset Q_j\subset  \bb(P_j,R)$.

The estimate \cite[(2.6)]{HNW} shows that each chart gives a
($j$-independent) metric equivalence
$|\psi_j(x)-\psi_j(y)| \sim \dist(x,y)$ and \cite[Lemma 3.2]{HNW}
shows that each $\psi_j$ provides a ($j$-independent) metric
equivalence between $W_2^k([-r,r]^d)$ and $W_2^k(Q_j)$.
\begin{itemize}
\item Let $Y_j = Q_j\cap Y$. By the triangle inequality, the fill
  distance of $Y_j$ in $Q_j$ satisfies
\begin{equation}
\label{square_fill}
h(Y_j,Q_j)\le 2 h_Y.\end{equation}
\item
 Let $\Upsilon_j = \psi_j(Y_j)$. Then by metric equivalence
\cite[(2.6)]{HNW},
 the fill distance of $\Upsilon_j$ in $[0,r]^d$ satisfies
\begin{equation}
\label{chart_fill}
h_j:=h(\Upsilon_j,[-r,r]^d) \sim h(Y_j,Q_j)\end{equation} 
with a $j$ independent constant.
\item
For $u\in W_p^k(\M)$, H{\"o}lder's inequality 
$\sum_{j=1}^K |a_j|\le K^{1/p'}\|a\|_{\ell_p}$ 
followed by monotonicity of the integral gives
\begin{equation}
\label{sob_mon}
\sum_{j=1}^K \|u\|_{W_p^k(Q_j)}
\le
K^{1/p'}\left(\sum_{j=1}^K \|u\|_{W_p^k(Q_j)}^p\right)^{1/p}
 \le 
 K \|u\|_{W_p^k(\M)}.
 \end{equation}
\item Similarly, for bounded $u$ (hence for any $u\in W_p^k(\M)$ with
  $k>d/p$), we have
\begin{equation}
\label{discrete_mon}
\sum_{j=1}^K \|u\left|_{Y_j}\right.\|_{\ell_p(Y_j)}
\le K\|u\left|_{Y}\right.\|_{\ell_p(Y)}.
\end{equation}
\end{itemize}
Using the cover by $Q_j$s and applying the metric equivalence gives
\begin{equation*}
\|u\|_{L_p(\M)}
\le 
\sum_{j=1}^K \|u\|_{L_p(Q_j)} 
\le 
C \sum_{j=1}^K \|u\circ\psi_j^{-1}\|_{L_p([-r,r]^d)} 
\end{equation*}
We now use \cite[3.5. Theorem]{Madych} on $[-r,r]^d$, to obtain, for each $j$,
that
$$
\|u\circ\psi_j^{-1}\|_{L_p([-r,r]^d)} \le C \left( h_j^{k}
  \|u\circ\psi_j^{-1}\|_{W_p^k([-r,r]^d)} + h_j^{d/p} \|
  (u\circ\psi_j^{-1}\left|_{\Upsilon_j})\right.
  \|_{\ell_p(\Upsilon_j)} \right).
$$
Thus,
$$
\|u\|_{L_p(\M)}
\le 
C 
\sum_{j=1}^K \left(
h_j^{k} \|u\circ\psi_j^{-1}\|_{W_p^k([-r,r]^d)}
+ 
h_j^{d/p} 
\| (u\circ\psi_j^{-1}\left|_{\Upsilon_j})\right.
\|_{\ell_p(\Upsilon_j)}
\right).
$$
By applying (\ref{chart_fill}) and (\ref{square_fill}), this gives
$$
\|u\|_{L_p(\M)} \le C \sum_{j=1}^K \left( h_Y^{k}
  \|u\circ\psi_j^{-1}\|_{W_p^k([-r,r]^d)} + h_{Y}^{d/p} \|
  (u\circ\psi_j^{-1}\left|_{\Upsilon_j})\right. \|_{\ell_p(\Upsilon_j)}\right)
$$
The metric equivalence \cite[Lemma 3.2]{HNW}  applied to 
$\|u\circ\psi_j^{-1}\|_{W_p^k([-r,r]^d)} $
along with the straightforward equality
$\| (u\circ\psi_j^{-1}\left|_{\Upsilon_j})\right.\|_{\ell_p(\Upsilon_j)}
=
 \| (u\left|_{Y_j})\right.\|_{\ell_p(Y_j)}$
gives
$$
\|u\|_{L_p(\M)}
\le C  \sum_{j=1}^K 
\left(
h_Y^{k} \|u\|_{W_2^k(Q_j)}+ h_{Y}^{d/p} \| (u\left|_{Y_j})\right.\|_{\ell_p(Y_j)}
\right).
$$
Finally, the estimates   (\ref{sob_mon}) and (\ref{discrete_mon})
provide
$$
\|u\circ\psi_j^{-1}\|_{L_p([-r,r]^d)} \le C \left( h_j^{k}
  \|u\circ\psi_j^{-1}\|_{W_p^k([-r,r]^d)} + h_j^{d/p} \|
  (u\circ\psi_j^{-1}\left|_{\Upsilon_j})\right.
  \|_{\ell_p(\Upsilon_j)} \right).
$$
and the result follows on taking $C_k:=2C$.
\end{proof}

\subsection{Norming sets for kernel spaces}\label{sets_kernel_spaces}

We discussed a variety of spaces involving SBFs in
Section~\ref{SBFs}. In particular, the SBF network, $S_X(\phi_s)$, for
the thin plate splines $\phi_s$, $s\in \mathbb N$, discussed in
Section~\ref{SBFs}, has a basis formed from Lagrange functions,
$\{\chi_\xi\}_{\xi \in X}$, $\chi_\xi(\eta)=\delta_{\xi,\eta}$,
$\xi,\eta\in X$. This basis satisfies the properties below, where
$N=\#X$
\begin{equation}\label{Riesz-Lagrange}
  C_{\mathfrak{L}}\, q_X^{d/2}\bigg(\sum_{\xi\in X} |a_\xi|^2\bigg)^{1/2}
\le
\|\sum_{\xi\in X} a_\xi \chi_\xi\|_{L_2(\sph^d)} 
\le
C_{\mathfrak{R}} \,
q_X^{d/2}\bigg(\sum_{\xi\in X}
|a_\xi|^2\bigg)^{1/2}.
\end{equation}
This was shown in \cite{FHNWW}. A basis satisfying these properties is
called a \emph{Riesz basis}. The identity holds for $L_p$ as well as
$L_2$; see equation \eqref{Riesz}.

If $S_X(\phi_s)$ is a subspace of $H^{k+2}$, then the Bernstein
inequality
$\|g\|_{H^{k+\epsilon}}\le C q_X^{-\gamma}\|g\|_{H^\epsilon}$
holds. If we take $\epsilon =2$ and $\gamma=2$, then we have
\[
\|g\|_{H^{k+2}}\le Cq_X^{-k} \|g\|_{H^2},
\]
which we will need below.

Since $S_X(\phi_s)$ is in $H^{k+2}$, its Lagrange basis,
$\{\chi_\xi, \xi\in X\}$, is a subset of $H^{k+2}$. From this and
\eqref{equiv}, it follows that the set $\{\opL \chi_\xi\}_{\xi\in X}$
is linearly independent and is a basis for the space
$S_X(\opL \phi_s)$. 

Suppose $g\in S_X(\phi_s)$, so that, by \eqref{equiv} and the
Bernstein inequality above,
\begin{equation}\label{bernstein_phi_s}
  \|g\|_{H^{k+2}} \le \Gamma_2\|\opL g\|_{H^k} \le \Gamma_1
  \|g\|_{H^{k+2}}\le Cq_X^{-k} \|g\|_{H^2}.
\end{equation}
The left side above implies that
$\|g\|_{H^2} \le \Gamma_2\|\opL g\|_{L_2}$. Combining this inequality
with the one above yields
$\|\opL g\|_{H^k} \le Cq_X^{-k}\|\opL g\|_{L_2}$. Since $X$ is quasi
uniform, $q_X \le CN^{-1/d}$. we have
\begin{equation}\label{norming_set_bernstein}
  \|w\|_{H^k}\le CN^{k/d}\|w\|_{L_2}, \forall\ w\in S_X(\opL \phi_s).
\end{equation}
This is the Bernstein inequality in \eqref{Bern}. Consequently,
Theorem~\ref{Theorem:Norming_Set} holds, yielding the following
result.

\begin{theorem}\label{norming_set_TPS}
  Let $w=\opL g$. Then, with $p=2$ and $Y$ as in
  Theorem~\ref{Theorem:Norming_Set},
\[
  \|w\|_{L_2} \le  Ch_Y^{d/2} \|w|_Y\|_{\ell_2(Y)}.
 \] 
\end{theorem}


\section{Stabilizing the Kansa matrix by oversampling}
In this section, we present a method to produce stable Kansa matrices
by strategic oversampling. The setting will be the sphere $\sph^d$ and
the kernels employed will be the thin-plate splines discussed in
Section~\ref{SBFs} and in the previous section. The results from
Theorem~\ref{norming_set_TPS}, and a suitable norming set, will imply
the Kansa matrix has a controlled lower bound.

Much of what we said previously for $\sph^d$ holds for a manifold
$\M$. Moreover, many of the proofs in Section~\ref{sets_kernel_spaces}
carry over \emph{mutatis mutandis} to the manifold case. When this
happens we will make note of it.

\paragraph{Lower bound of the Kansa matrix}
We will now show how to construct a stable asymmetric collocation
matrix, given a kernel $\Phi$, an operator $\opL$ as defined in
Section~\ref{opL}, and a point set $X\subset \M$.

\paragraph{Kansa matrix with alternative bases}
If we consider a general basis $\{B_k, 1\le k\le N\}$ for the kernel
space $S_X(\Phi)$, where $\Phi$ may be an SCPD kernel, the Kansa
method has the Vandermonde-like structure:
\[
\K := \Bigl(\opL B_k (y_j)\Bigr)_{j,k}.
\]
Although using bases other than the standard $\phi(x\cdot y_j)$ causes
the coefficient vector $a$ obtained from $\K a= f|_{Y}$ to change, the
kernel network $v\in S_X(\Phi)$ which solves (\ref{Kansa}) remains
invariant.

This flexibility has two immediate benefits.  It allows us to easily
consider {\em conditionally positive definite} kernels on $\sph^d$,
specifically the thin-plate spline spaces $S_X(\phi_s)$ and
$S_X(\opL\phi_s)$, where the bases are not just rotations of the
kernel $\phi_s$ or $\opL\phi_s$. They contain polynomial parts. Being
able to use different spaces also allows us to choose bases for
them. For example, the Lagrange bases $\{\chi_\xi, \xi\in X\}$ for
$S_X(\phi_s)$ give well conditioned matrices. This permits us to
separate the stability of the Kansa method from the potentially poor
conditioning of the basis.

\paragraph{Stability ratio} For a given basis
$\{B_k, 1\le k\le N\}$ for $S_X(\Phi)$, we define the {\em
  stability ratio}
\[
  \mathsf{r}_2 (X):= \max \left. \left\{ \frac{
        \|a\|_{\ell_2(X)}}{\|g\|_{L_2(\M)}} \;\bigg{|}\, g = \sum a_k B_k
      \in S_X(\Phi)\right\}\right. .
\]
This is a quantity which has been introduced and studied on spheres
in \cite[(1.1)]{MNPW}
for the kernel basis 
$B_k = \Phi(\cdot,x_k)$.
There it has been shown that 
$ \mathsf{r}_2(X) \sim q^{d/2 -2m}$ for many kernels 
$\Phi:\Sph^d\times\Sph^d\to \R$ 
having Sobolev native space $\mathcal{N}(\Phi) = H^2(\Sph^d)$.

If $\{B_k, 1\le k\le N\}$ is a Riesz basis for $S_X(\Phi)$ in the sense
that for $u = \sum_{k=1}^N a_k B_k\in S_X(\Phi)$ the estimate
\begin{equation}\label{Riesz}
c_{\mathfrak{L}} q^{d/p}\left(\sum_{k=1}^N |a_k|^p\right)^{1/p} \le
\|\sum_{k=1}^N a_k B_k\|_{L_p(\M)}\le
C_{\mathfrak{R}} q^{d/p}\left(\sum_{k=1}^N |a_k|^p\right)^{1/p}.
\end{equation}
holds, then
$\frac{1}{c_{\mathfrak{L}}}\sqrt{N}\le \mathsf{r}_2 (X) \le
\frac{1}{C_{\mathfrak{R}}}\sqrt{N}]$ and thus can be controlled by
$ q^{-d/2}$.  The existence of Riesz bases for certain kernel spaces
has been demonstrated in \cite{HNSW, HNRW2}.  For spheres, the
restricted thin plate splines are shown to have this property in
\cite{FHNWW}.

Under the preceding assumptions we have the following result, which
assumes that $Y$ is a norming set.
\begin{lemma}
\label{norming_set_lemma}
If $\Phi$ is a positive definite kernel on $\M$, $X\subset \M$
is a point set, $\{B_k, 1\le k  \le N\}$ is a basis for $S_X(\Phi)$ having
stability ratio $\mathsf{r}_2(X)$, and $Y$ is a norming set for
$S_Y(\opL^{(1)}\Phi)$ with cardinality $M=\#Y$, then
$$
\|\K a\|_{\ell_2(Y)} 
\ge  
\frac{1}{\mathsf{r}_2 (X)}
\frac{1}{C_{\mathsf{N}}\Gamma_2} \sqrt{\frac{M}{2}} \|a\|_{\ell_2( X)}
$$
\end{lemma}
\begin{proof}
  From (\ref{norming_set}) we have, with $w= \sum a_k \opL B_k$, that
  $ \|\K a\|_{\ell_2(\tilde{Y})} = \|\sum_{k=1}^N a_k \opL
  B_k(\cdot)\|_{\ell_2(\tilde{Y})} \ge \frac{1}{C_{\mathsf{N}}}\sqrt{
    \frac{{M}}{2} }\|w\|_{L_2} $.  By \eqref{equiv},
  $\|w\|_{L_2} \ge \Gamma_2^{-1}\| \sum a_k B_k\|_{H_2}\ge
  \Gamma_2^{-1}\| \sum a_k B_k\|_{L_2}$, so
  $\|\K a\|_{\ell_2(\tilde{Y})} \ge
  \frac{1}{C_{\mathsf{N}}\Gamma_2}\sqrt{\frac{M}{2}} \| \sum a_k
  B_k\|_{L_2}.$ The result follows from the definition of the
  stability ratio.
\end{proof}

In particular, if $\{B_k\}_{k=1}^N$ is a Riesz basis, then for $p=2$,
$q_Y^{d/2}\sim M^{-1/2}$ and $C_{\mathsf{N}}\sim C_{\mathfrak{R}}$. In
addition, $\mathsf{r}_2(X) \ge \frac{1}{c_{\mathfrak{L}}}\sqrt{N}$ so
$\K$ is bounded below by
\[
  \|\K a \|_{\ell_2(Y)} \ge C \sqrt{\frac{M}{N}}\|a\|_{\ell_2(X)};
\]
i.e., the lower bound is proportional to the square root of the degree
of oversampling, $\kappa$; see Section~\ref{norming_set_Y}.

We have the following result, which follows from
Theorem~\ref{Theorem:Norming_Set}:
\begin{theorem}\label{lower_bnd_G}
  If $\opL$ satisfies \eqref{equiv}, $S_X(\opL^{(1)}\Phi)$ satisfies
  the Riesz basis property \eqref{Riesz}, in the sense that the family
  $(\opL B_k)$ satisfies \eqref{Riesz}, and $S_X(\opL \Phi)$ satisfies the
  Bernstein inequality
\begin{equation}\label{Bernstein}
  (\forall w\in   S_X(\opL^{(1)} \Phi))\qquad 
  \| w\|_{W_{2}^k(\M)} 
  \le 
  C_{\mathfrak{B}}  h_X^{-k} \|w\|_{L_{2}(\M)}.
\end{equation}
then there is a quasi uniform point set $Y\subset \M$ with
$h_Y\sim h_X$ for which
\begin{equation}\label{norm_Ka_lower_bnd}
\|\K a\|_{\ell_2(Y)} 
\ge  
\frac{ c_{\opL} c_{\mathfrak{R}}}{C_{\mathsf{N}}}
\sqrt{\kappa/2}
\|a\|_{\ell_2( X)}.
\end{equation}
\end{theorem}

The matrix $\K$ plays an important role in least squares
approximation. Let $G:=\K^*\K$. Since $\K:\ell_2(X)\to \ell_2(Y)$,
$G:\ell_2(X) \to \ell_2(X)$. Note that
$\|\K a\|^2_{\ell_2(Y)} \ge C\kappa \|a\|_{\ell_2( X)}^2$, where
$C=\frac{1}{2}\frac{ c_{\opL}^2
  c_{\mathfrak{R}}^2}{C_{\mathsf{N}}^2}$. Of course,
$\|\K a\|^2_{\ell_2(Y)}=\langle \K a, \K a\rangle_{\ell_2(Y)}= \langle
\K^*\K a, a\rangle_{\ell_2(X)}\ge C\kappa \|a\|_{\ell_2( X)}^2 $. It
follows immediately that $G=\K^*\K$ is invertible and that
\begin{equation}\label{G_inv}
\|G^{-1}\|_{\ell_2(X)}\le C^{-1}\kappa^{-1}.
\end{equation}

\section{Solution via least squares}\label{soln_via_ls}
We assume that there are constants $\rho^*$, $C$ and $c$ so that the
following holds:
given a quasi uniform set $X\subset \Sph^d$ 
with $\#X=N$, and $Y\subset \Sph^d$ with $\#Y=M$ which satisfies:
\begin{itemize}
\item $\rho_Y\sim \rho_X$ in the sense that, there exists a global
constant $\rho^*$ so that both $\rho_Y $ and $ \rho_X$ are less than
$\rho^*$
\item $Y$ is an $L_2$ norming set for $S_X(\opL^{(1)}\Phi)$ as
  considered in Section \ref{S:Norming}. Namely,
  $\|w\|_{L_2} \le {C}{M}^{-1/2} \|w|_Y\|_{\ell_2(Y)}$, where
  $M \sim Cq_Y^{-d}$.
\item The sets $Y$ and $X$ are comparable in the sense that
  $cq_X\le q_Y<q_X$, or equivalently $ch_X\le h_Y<h_X$.
\end{itemize}

Consider the (rectangular) Kansa matrix
$\K =\bigl(\opL \chi_j(y_k)\bigr)$.  We attempt to solve $\opL u= f$
by Kansa's method, with and $u^* = \sum_{j=1}^N a_j \chi_j$ and
coefficients $\vec{a}=(a_j)$ obtained from $\K \vec{a}=f|_Y$. Since
this system is over determined, its solution is obtained by discrete
least squares, with $\vec{a} = (\K^* \K)^{-1} \K^* (f|_Y)$.

Let $u^* = \sum a_j \chi_j$, where the $a_j$'s are components of
$\vec{a}$. For the true solution $u$ of $\opL u=f$, we have
 \begin{equation}\label{u_u*_split}
   \|u-u^*\|_{L_2} \le \|u-I_X u\|_{L_2} + \| I_Xu - u^*\|_{L_2}.
 \end{equation}
 The former is easily bounded by $Cq_X^{2s+d}\|u\|_{H_{2s+d}}$,
 provided the SBF is the thin-plate spline $\phi_s$, with $s\in \N$
 and $u\in H^{2s+d}(\Sph^d)$, as the result below shows.

 \begin{proposition}\label{est_u_interp}
  Suppose that $u\in H^{2s+d}(\Sph^d)$ and that the thin-plate spline
  $\phi_s$, $s \in \N$, is the SBF used. Then,
\[
  \|u-I_X u\|_{H^\beta}\le
  C\rho^{2s+d-\beta}q_X^{2s+d-\beta}\|u\|_{H^{2s+d}}.
\]
\end{proposition}

\begin{proof} If $X$ is quasi uniform, then for $\phi_s$ the
  coefficients $\hat \phi_s(\ell)$ satisfy
  $\hat \phi(\ell) \sim \ell^{2s+d}$, if $\ell>s$. The result then
  follows from \cite[Theorem A.3]{NRW}, with $\beta \le 2s+d$ and
  $2\tau=2s+d$.
\end{proof}

For the latter, we begin by letting $g=\opL(u-I_Xu)$, and considering
the TPS $\phi_{s-1}$ and the norming set $Y$ discussed in
Lemma~\ref{norming_set_lemma}. A simple application of the triangle
inequality implies that
\[
\|I_Yg\|_{L_2} \le \|g\|_{L_2}+ \|I_Yg-g\|_{L_2}.
\]
Since $\|g\|_{L_2}=\|\opL(u-I_Xu)\|$, \eqref{equiv} gives
$\|g\|_{L_2}\le \|u-I_Xu\|_{H^2}$, and then applying Proposition~
\ref{est_u_interp} for $\beta =2$ results in this estimate:
\begin{equation}\label{g-est}
\|g\|_{L_2}\le C\rho_X^{2s+d-2}q_X^{2s+d-2}\|u\|_{H^{2s+d}}.
\end{equation}
Applying the same proposition to $ \|I_Yg-g\|_{L_2}$, this time for
$\phi_{s-1}$, yields
\begin{equation}\label{I_Yg-g-est}
  \|I_Yg-g\|_{L_2}\le C\rho_Y^{2s+d-2}q_Y^{2s-2+d}\|g\|_{H^{2s+d-2}}.
\end{equation}
Combining the inequalities \eqref{g-est} and \eqref{I_Yg-g-est} above
and again using \eqref{equiv} and the fact that $q_X \sim q_Y$, we
arrive at the result below.
\begin{equation}\label{I_Yg-est}
  \|I_Yg\|_{L_2} \le C\rho^{2s+d-2}q_Y^{2s+d-2}\|u\|_{H^{2s+d}}.
\end{equation}

The next step is to use \eqref{Riesz-Lagrange}, which applies since
the Lagrange basis for $S_Y(\phi_{s-1})$ is a Riesz basis. Letting
$\{\widetilde \chi_\eta\}_{\eta \in Y}$ be that basis, we have
$I_Yg=\sum_{\eta\in Y}g(\eta)\widetilde \chi_\eta$. Consequently,
applying \eqref{Riesz-Lagrange} results in:
\begin{equation}\label{Riesz-Lagrange-2}
  C_{\mathfrak{L}}\, q_Y^{d/2}\bigg(\sum_{\xi\in Y} |g(\eta)|^2\bigg)^{1/2}
  \le
  \|I_Yg\|_{L_2(\sph^d)} 
  \le
  C_{\mathfrak{R}} \,
  q_Y^{d/2}\bigg(\sum_{\eta\in Y}
  |g(\eta)|^2\bigg)^{1/2}.
\end{equation}
This and \eqref{I_Yg-est} imply that
\[
\|g|_Y\|_{\ell_2} \le C\rho^{2s+d-2}q_Y^{2s+d/2-2}\|u\|_{H^{2s+d}}.
\]

The final step is to note that
$g(\eta)=\opL(I_Xu)(\eta)-(\opL u)(\eta) =\sum_{\xi\in
  X}u(\xi)\opL(\chi_\xi)(\eta)-f(\eta)$. In terms of
$\K_{\eta,\xi}=\opL(\chi_\xi)(\eta)$ we see that
\[
  \|\K(\underbrace{I_Xu|_X}_{u|_X})|_Y -f |_Y\|_{\ell_2}\le
  C\rho^{2s+d-2}q_Y^{2s+d/2-2}\|u\|_{H^{2s+d}}.
\]

\begin{lemma}
  Let $\vec{a}\in \ell_2(X)$ satisfy
  $\|\K\vec{a}-f|_Y\|_{\ell_2(Y)}=\min_{\alpha  \in \ell_2(X)} \|\K \alpha
  -f|_Y\|_{\ell_2(Y)}$. Then,
\[
  \|\K\vec{a}-f|_Y\|_{\ell_2(Y)} \le
  C\rho^{2s+d-2}q_Y^{2s+d/2-2}\|u\|_{H^{2s+d}}.
\]
\begin{proof}
  Since we may take $\alpha=I_Xu|_X=u|_X$, the left side above cannot
  exceed the right side of the previous estimate.
\end{proof}
\end{lemma}

Note that
\[
  \|\K\vec{a}-f|_Y+
  f|_Y-\K(u|_X)\|_{\ell_2(Y)}\le \|\K\vec{a}-f|_Y\|_{\ell_2(Y)}+
  \|f|_Y-\K(u|_X)\|_{\ell_2(Y)}.
\]
Consequently,
\[
  \|\K\vec{a}-\K(u|_X)\|_{\ell_2(Y)} \le
  C\rho^{2s+d-2}q_Y^{2s+d/2-2}\|u\|_{H^{2s+d}}.
\]
In the last inequality, using \eqref{norm_Ka_lower_bnd}, with $a$
replaced by $\K\vec{a}-\K(I_Xu|_X)$, and noting that we are using a
Riesz basis, we arrive at
\[
  \|\K(\vec{a}-u|_X)\|_{\ell(Y)}\ge
  C\sqrt{\frac{M}{N}}\|\vec{a}-u|_X\|_{\ell_2(X)}.
\]
Since the sets $X$ and $Y$ are comparable in the sense that
$q_X\sim q_Y$ and $h_x\sim h_Y$, $M\sim N$. 
\begin{equation}\label{est_norm_a-u_X}
  \|\vec{a}-u|_X\|_{\ell_2(X)} \le
 C \|\K\vec{a}-\K(u|_X)\|_{\ell_2(Y)} \le
  C\rho^{2s+d-2}q_Y^{2s+d/2-2}\|u\|_{H^{2s+d}}.
\end{equation}

Returning to estimating $\|I_X u - u^*\|_{L_2}$, we have that
$I_X u - u^* = \sum_{j=1}^N (u(x_j) - a_j) \chi_j$, which we can bound
using the fact that $\{\chi_j\}$ is a Riesz basis. So, by
\eqref{Riesz},
$ \| I_Xu - u^*\|_{L_2}\le C q_X^{d/2}\|\vec{a} -
u|_X\|_{\ell_2(X)}$. This and the previous inequality, with $q_Y\sim q_X$,
imply that
\begin{equation}\label{u_X-u*-est}
 \| I_Xu - u^*\|_{L_2}\le
  C\rho^{2s+d-2}q_X^{2s+d-2}\|u\|_{H^{2s+d}}.
\end{equation}

\begin{theorem}\label{L2_est_L_2_proj} Let $u$ solve $\opL u= f$, with
  $u,f\in H^{2s+d}$. If the SBF is the thin-plate spline $\phi_s$,
  $s \in \N$, defined in \eqref{TPS}, then
\[
  \|u-u^*\|_{L_2} \le C\rho^{2s+d}q_X^{2s+d-2} \|u\|_{H^{2s+d}}.
\]
\end{theorem}
\begin{proof}
  As we noted at the start of this section,
  $\|u-u^\ast\|_{L_2}\le \|u-I_Xu\|_{L_2}+\|I_Xu-u^\ast\|_{L_2}$. By
  Proposition~\ref{est_u_interp}, the interpolation error estimate is
  comparable to the $\|I_Xu-u^\ast\|_{L_2}$. The result then follows.
\end{proof}

\section{Square system, with thinning}\label{thinning algorithm}

We now consider the asymmetric matrix
$\K :=\Bigl(\opL \chi_{k} (y_j)\Bigr)_{j,k}\in \R^{M\times N}$. Our
goal is to ``thin'' the point set
$Y\to \widetilde{Y}= \{\tilde{y}_j\mid 1\le j\le N\}$, where
$\#\tilde{Y}=N$, so that
$\Bigl(\opL \chi_{k} (\tilde{y}_j)\Bigr)_{k,j}\in \R^{N\times N}$ is a
relatively stable $N\times N$ matrix. To do this we will use the QR
rank reducing (RRQR) factorization from \cite{Gu-Eisenstat}; we
discuss this below.

To begin, note that the $N^{th}$ singular value of $\K$ is
$\sigma_N(\K) = \inf_{\|\alpha\|_{\ell_2}=1}\|\K\alpha\|_{\ell_2(Y)}$.
It follows from this observation and \eqref{norm_Ka_lower_bnd} that
\begin{equation}\label{s-val_lower_bnd_K}
  \sigma_N(\K) \ge C\sqrt{\kappa}, \ \kappa=M/N.
\end{equation}


Choose $Y$ so that $M$ is a multiple of $N$ (if necessary, enlarge
$Y$),so that $\kappa=M/N$ is an integer $2$ or larger. Let
$e_\kappa=(1\ 1 \cdots 1)\in \R^{1\times \kappa}$ and define the
$M\times M$ partitioned matrix consisting of $\kappa$ copies of $\K$:
 \[
   \widetilde \K := \Bigl( \K \ \vert \ \K \ \vert \ \dots \ \vert \
   \K\Bigr) = e_\kappa\otimes \K,
\]
where $e_\kappa\otimes \K$ is the Kronecker product of $e_\kappa$ with
$\K$. We want to use this product to find the singular values of
$\widetilde \K$.

For any two matrices $A$ and $B$, the singular values of $A\otimes B$
are the entries in $\Sigma(A)\otimes \Sigma(B)$
\cite[pg.~294]{Van_Loan}; that is, if $\sigma_i(A)$ is a singular
value of $A$, $i\le \rank(A)$ and $\sigma_j(B)$ are those for $B$,
$j\le \rank(B)$, then those for $A\otimes B$ are
$\sigma_i(A)\sigma_j(B)$. Because the only singular value of
$(1\ 1 \cdots 1)$ is $\sigma_1(e_\kappa)=\sqrt{\kappa}$ and those for $\K$ are
$\sigma_j(\K)$, with $1\le j \le N=\rank(\K)$, it follows that the
singular values of $\widetilde \K$ satisfy
\begin{align}\label{s-vals_big_K}
  \sigma_j(\widetilde \K) &=\sqrt{ \kappa}  \sigma_j(\K),\ 1\le j\le N.
                            \nonumber\\
  \sigma_N(\widetilde \K) &=\sqrt{ \kappa}\sigma_N(\K)\ge C\kappa^{3/2},
\end{align}
where the last inequality follows from \eqref{s-val_lower_bnd_K}.

\paragraph{Rank revealing factorization} We will give a brief
discussion of the rank revealing QR factorization (RRQR) discussed in
\cite{Gu-Eisenstat} for an $m\times n$ matrix $F$, with $m\ge n$. The
matrix $F$ can be factored as follows:
\begin{equation}\label{rrqr}
  F\Pi = Q\begin{pmatrix} A_k &B_k\\0& C_k\end{pmatrix}.
\end{equation}
The matrix $\Pi$ is an $n\times n$ permutation matrix, $Q$ is an
$m\times m$ orthogonal matrix, $A_k$ is a $k\times k$ upper triangular
matrix with non-negative diagonal elements. The remaining matrices
$B_k$ and $C_k$ are, respectively, $k\times (n-k)$ and
$(m-k)\times (n-k)$. The factorization is called \emph{rank revealing}
if $\sigma_{min}(A_k) \ge \sigma_k(F)/q_1(k,n)$ and
$\sigma_{max}(C_k)\le q_1(k,n)\sigma_{k+1}(F)$, where $q_1(k,n)$ is
bounded above by a low degree polynomial. If in addition,
$|(A_k^{-1}B_k)|\le q_2(k,n)$, where $q_2(k,n)$ is also bounded above
by a low degree polynomial, then it is called a \emph{strong} RRQR. In
\cite[Sect.~3]{Gu-Eisenstat}. Gu and Eisenstat show that there is a
permutation $\Pi$ such that the factorization \eqref{rrqr} is a strong
RRQR, with $q_1=\sqrt{1+k(n-k)}$ and $q_2=1$.

Recall that $(A\otimes B)^T=A^T\otimes B^T$, so
$\widetilde\K^T= e_\kappa^T\otimes\K^T \in \RR^{M\times M}$, where
$M=\kappa N$, $\kappa\ge 2$. In \eqref{rrqr}, choose $k=N<M$; the
factorization then becomes
\begin{equation}\label{factor_with_Pi}
\widetilde{\K}^T \Pi = Q\begin{pmatrix} A_N &B_N\\0& C_N\end{pmatrix}.
\end{equation}

We will need the singular values of $\widetilde{\K}^T$. Because a
matrix and its transpose have the same singular values, by
\eqref{s-vals_big_K} we see that
$\sigma_j(\widetilde \K^T)=\sigma_j(\widetilde
\K)=\sqrt{\kappa}\sigma_j(\K) $. Since a strong RRQR exists for
$\widetilde{\K}^T$, we have that
\[
  \sigma_{min}(A_N)=\sigma_N(A_N)\ge
  \sigma_N(\widetilde{\K}^T)/q_1(N,M)=
  \sqrt{\kappa}\sigma_N(\K)/q_1(N,M),
\]
where $\kappa=M/N$,
$q_1(N,M)=\sqrt{1+N(M-N)}=\sqrt{1+N^2(\kappa-1)}\sim N$. This and
\eqref{s-vals_big_K} imply that
\begin{equation}\label{lower_bnd_sigma_A_N}
  \sigma_N(A_N)\ge  C N^{-1} \kappa^2.
\end{equation}

Returning to the factorization above, $\Pi$ permutes the columns of
$\widetilde \K^T$. If we view these columns as labeled by $y_j$'s, the
permutation effectively changes these to the $\widehat
y_j$'s. Assuming this has been done, we may drop $\Pi$ in
\eqref{factor_with_Pi}. Thus the $j^{th}$ column in
$\widetilde{K}^T\Pi$ is now $\opL\chi_k(\widehat{y}_j)$, where the row
index $k=1\ldots N$ is repeated $\kappa$ times. In addition, by
dropping the columns from $N+1$ to $M$ in the resulting equation, we
form a reduced $M\times N$ version of \eqref{factor_with_Pi},
\[
\widetilde \K^T_{red}=  Q\begin{pmatrix} A_N \\ 0\end{pmatrix},
\]
where the rows of the reduced matrix are $\kappa$ copies of the matrix
$\K_{red}^T$, which is the matrix $\K^T$ with the appropriate columns
removed. Thus, $\widetilde \K^T_{red}=e_\kappa^T\otimes \K_{red}^T$;
hence, $\sigma_N(\widetilde
\K^T_{red})=\sqrt{\kappa}\sigma_N(\K_{red}^T)$. In addition, since the
singular values of a matrix are invariant under left and/or right
multiplication by an orthogonal matrix, we see that
\[
  \sigma_N(\widetilde \K^T_{red})=\sqrt{\kappa} \sigma_N(\K^T_{red})=
  \sigma_N\bigg( Q\begin{pmatrix} A_N \\
    0\end{pmatrix}\bigg)= \sigma_N\begin{pmatrix} A_N \\
    0\end{pmatrix}=\sigma_N(A_N),
\]
so $\sigma_N(\K_{red}^T) = \sigma_N(A_N)/\sqrt{\kappa}$. Moreover, the
previous equation, $\sigma_N(\K_{red}^T) =\sigma_N(\K_{red})$ and
\eqref{lower_bnd_sigma_A_N} imply
\begin{equation}\label{est_least_sigma}
  \sigma_N(\K_{red})\ge C N^{-1} \kappa^{3/2}.
\end{equation}
This, coupled with the singular value decomposition for $\K_{red}$,
gives us this result.

\begin{proposition}\label{thinned_kansa_matrix} Let $Y$ be a set of
points satisfying the properties listed for a norming set in
Section~\ref{norming_set_Y}, possibly extended to have $\#Y=\kappa \#
X$, where $\kappa$ is an integer larger than or equal to 2. Then there
exists a $\widetilde Y \subset Y$, with $\# \widetilde Y=\#X =N $ such
that the $N\times N$ matrix $\K_{red}$, with $\widehat y_j$'s
replacing the first $N$ $y_j$'s in the Kansa matrix $\K$, satisfies $
\|\K_{red}a\|_{\ell_2(\widetilde Y)} \ge CN^{-1}\|a\|_{\ell_2(X)}, $
and is invertible, with $\| \K_{red}^{-1}\|_{\ell_2(X)} \le CN$.
\end{proposition}

\paragraph{Error estimates} Suppose the conditions in
Theorem~\ref{L2_est_L_2_proj} hold, and that $\widetilde Y$ and
$\K_{red}$ are as in Proposition~\ref{thinned_kansa_matrix}. The lower
bound
$\|\K_{red}a\|_{\ell_2(\widetilde Y)} \ge CN^{-1}\|a\|_{\ell_2(X)}$
plays the role of \eqref{norm_Ka_lower_bnd} for the case at
hand. Replacing \eqref{norm_Ka_lower_bnd} by it, carrying out the
calculations in Section~\ref{soln_via_ls} and using the same argument
from that section here yields
$\|u-u^*\|_{L_2} \le C\rho^{2s+d}q_X^{2s-2} \|u\|_{H^{2s+d}}$.  Since
the interpolation error discussed earlier has order $q_X^{2s+d}$, it
won't contribute to the error for $\|I_X-u^*\|_{L_2}$ derived
above. Consequently, our final estimate is given below:

\begin{theorem}\label{thinning}
\[
  \|u-u^*\|_{L_2} \le C\rho^{2s+d}q_X^{2s-2} \|u\|_{H^{2s+d}}.
\]
\end{theorem}

\begin{remark}
  \rm Although a norming set is needed for the proof of the error
  estimate in Theorem~\ref{thinning}, in practice one can use any set
  $Z$ to replace $\widetilde Y$, provided where $|Z| = |X|$ and
  $q_Z \sim q_X$. However, there may be a price to be paid. If we also
  have $\|\K_{red}^{-1}\|_{\ell_2(Z)}\le CN^\alpha$, with
  $\alpha > 1$, then from \eqref{est_norm_a-u_X},
  \[
  \|\vec{a}-I_Xu|_X\|_{\ell_2(Z)} \le \|\K_{red}^{-1}\|_{\ell_2(Z)}
  \|\K_{red}\vec{a}-\K_{red}I_Xu|_X)\|_{\ell_2(Z)}.
\]
Since $N^\alpha\sim q_X^{-\alpha d}\sim q_Z^{-\alpha d}$, this implies
$ \|\vec{a}-I_Xu|_X\|_{\ell_2(Z)} \le Cq_X^{-\alpha
  d}\|\K_{red}\vec{a}-\K_{red}I_Xu|_X)\|_{\ell_2(Z)}$.  Following the
argument leading up to Theorem~\ref{L2_est_L_2_proj} , we have
$ \| I_Xu - u^*\|_{L_2}\le Cq_X^{2s-(\alpha-1) d-2}\|u\|_{H^{2s+d}}
$.\rm
\end{remark}

As a final comment, we note that there is room to improve
Theorem~\ref{thinning}. This comes directly from the cost of the
thinning method, specifically the that $q_1(N,M)\sim N$.  Of course,
this could be addressed by a better performing rank revealing qr
method (for which $q(N,M)\ll N$), although there may also be thinning
methods which are more specifically suited to kernels. For instance,
it may be possible to modify the greedy, symmetric kernel collocation
method presented in \cite[Section 4.2]{HSW}.

\section*{Acknowledgment} The authors wish to thank Professor Rachel
Ward for suggesting the paper by Gu and Eisenstat \cite{Gu-Eisenstat}.

\bibliographystyle{plain}

\end{document}